\documentclass[12pt]{amsart}
 \usepackage[latin1]{inputenc}      
\usepackage{tikz}
\usetikzlibrary{calc,through,backgrounds}
\usetikzlibrary{snakes}
\usetikzlibrary{arrows}
\usepackage{amssymb}
\usepackage{amsmath}
\usepackage{enumitem}

\newcommand{\Z}{{\mathbb Z}}

\newcommand{\Q}{{\mathbb Q}}

\newcommand{\EEE}{\mathcal{E}}
\newcommand{\FFF}{\mathcal{F}}
\newcommand{\F}{\FFF^{\Z^2}}

\newcommand{\e}{\varepsilon}
\newcommand{\D}{\hat{D}}

\newtheorem{theo}{Theorem}
\newtheorem{prop}[theo]{Proposition}

\newtheorem{lemma}[theo]{Lemma}

\newenvironment{proofof}[1]{\noindent {\bf Proof of #1.}}{ \hfill\qed\\ }

\begin{document}

\title{Typical recurrence for the Ehrenfest wind-tree model}

\author{Serge Troubetzkoy}
\address{Institut de Mathématiques de Luminy, 163 avenue de Luminy, Case
907, 13288 Marseille Cedex 9, France}
\email{troubetz@iml.univ-mrs.fr}

\begin{abstract}
We show that the typical wind-tree model, in the sense of Baire,  is  recurrent and has a dense
set of periodic orbits.  The recurrence result also 
holds for the Lorentz gas : the typical  Lorentz gas, in the sense of Baire, is
recurrent.  These Lorentz gases need not be of finite horizon!
\end{abstract}

\maketitle \markboth{Serge Troubetzkoy}{Ehrenfest wind-tree model}

In 1912 Paul and Tatiana Ehrenfest proposed the wind-tree model of
diffusion in order to study the statistical interpretation of the
second law of thermodynamics and the applicability of the Boltzmann
equation \cite{EhEh}.  In the Ehrenfest wind-tree model, a point
particle (the ``wind'') moves freely on the plane and collides with the usual law
of geometric optics with randomly placed fixed square scatterers
(the ``trees'').  The notion of ``randomness'' was not made precise, in fact it would have
been impossible to do so before 
Kolmogorov laid the foundations of probability theory in the 1930s.
We will call the subset of the plane obtained by removing
the obstacles the billiard table, and the the motion of the point the billiard flow.

From the mathematical rigorous point of view, there have been two
results on recurrence for wind-tree models, both on a periodic version where 
the
scatterers are identical rectangular obstacles located periodically
along a square lattice on the plane, one obstacle centered at each
lattice point.  Hardy and Weber \cite{HaWe}
proved recurrence and abnormal
diffusion of the billiard flow for special dimensions of the obstacles
and for very special directions, using results on skew products above
rotations. More recently Hubert, Lelièvre, and Troubetzkoy have
studied the general full occupancy periodic case \cite{HuLeTr}.  They proved that, if the lengths of the sides
of the rectangles belong to a certain dense $G_{\delta}$ subset $\EEE'$, then the
dynamics is recurrent and they gave a lower bound on the diffusion rate.
The recurrence was proven by  analysis of the case when the lengths of the sides are rational
and belong to the set
\begin{align*}
\EEE & = 
\bigl\{\, (a,b) = (p/q,r/s) \in \Q \times \Q :
\\
& \phantom{={}}\quad
(p,q) = (r,s) =1,\quad 0<p<q,\quad 0<r<s,
\\[-3pt]
& \phantom{={}}\quad
p,r \textrm{ odd,}\quad q,s \textrm{ even}
\,\bigr\}.
\end{align*}
Another set of rationals lengths was used to prove the diffusion result. The set $\EEE'$ consists of 
 dimensions which are sufficiently well approximable by both of these rational sets. 

In this article we will prove the recurrence of  random wind-tree models.  We consider the following model.
Fix a finite or countable set of dimensions of obstacles $\FFF \subset (0,1)^2 \cup \{(0,0)\}$ such that  $\FFF \cap (\EEE \cup \EEE') \ne \emptyset$.  Let $e$ be in this intersection and  let $W_e$ denote the  billiard table  with identical obstacle $e$ at each lattice site.
Consider the set of all wind-tree models $\F$ with the product topology on $\F$.  
A lattice site with an obstacle of dimension $(0,0)$ will be interpreted as a lattice site without obstacle.

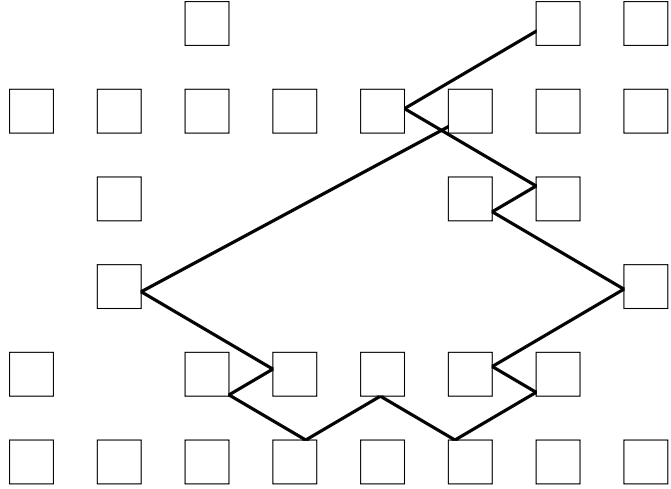
\begin{figure}[ht]
\begin{minipage}[ht]{0.58\linewidth}
\centering
\begin{tikzpicture}
  [x=0.0972222222222cm,y=0.0571895424837cm,scale=0.6]

\draw[thin](20,0) rectangle +(10,17);
\draw[thin](140,0) rectangle +(10,17);
    
     \foreach \i in {0,20,...,140}
    \draw[thin] (\i,-68) rectangle +(10,17);
    
      \foreach \i in {0,20,...,140}
    \draw[thin] (\i,68) rectangle +(10,17);
    \draw[thin] (0,-34) rectangle +(10,17);
    \draw[thin] (20,34) rectangle +(10,17);
    \draw[thin] (40,-34) rectangle +(10,17);
    \draw[thin] (60,-34) rectangle +(10,17);
    \draw[thin] (80,-34) rectangle +(10,17);
    \draw[thin] (100,-34) rectangle +(10,17);
    \draw[thin] (120,-34) rectangle +(10,17);
    \draw[thin] (120,102) rectangle +(10,17);
    \draw[thin] (120,34) rectangle +(10,17);
     \draw[thin] (140,102) rectangle +(10,17);
   \draw[thin] (100,34) rectangle +(10,17);
    \draw[thin] (40,102) rectangle +(10,17);
   \draw[very thin,fill] (100,70) -- +(0,1) -- (30,7) -- +(0,-1) -- cycle;
  \draw[very thin,fill] (30,6) -- +(0,1) -- (60,-23) -- +(0,-1) -- cycle;
  \draw[very thin,fill] (50,-34) -- +(0,1) -- (68,-51) -- +(-1,0) -- cycle;
  \draw[very thin,fill] (84,-34) -- +(1,0) -- (102,-51) -- +(-1,0) -- cycle;
  \draw[very thin,fill] (120,-33) -- +(0,1) -- (110,-22) -- +(0,-1) -- cycle;
  \draw[very thin,fill] (140,7) -- +(0,1) -- (110,38) -- +(0,-1) -- cycle;
  \draw[very thin,fill] (120,47) -- +(0,1) -- (90,78) -- +(0,-1) -- cycle;

  \draw[very thin,fill] (60,-24) -- +(0,1) -- (50,-33) -- +(0,-1) -- cycle;
  \draw[very thin,fill] (67,-51) -- +(1,0) -- (85,-34) -- +(-1,0) -- cycle;
  \draw[very thin,fill] (101,-51) -- +(1,0) -- (120,-33) -- +(0,1) -- cycle;
  \draw[very thin,fill] (110,-23) -- +(0,1) -- (140,8) -- +(0,-1) -- cycle;
  \draw[very thin,fill] (110,37) -- +(0,1) -- (120,48) -- +(0,-1) -- cycle;
  \draw[very thin,fill] (90,77) -- +(0,1) -- (120,108) -- +(0,-1) -- cycle;

\end{tikzpicture}
\end{minipage}
\caption{A random wind-tree model with obstacles of sizes $(1/2,1/2)$ and $(0,0)$.}
\label{fig:escaping}
\end{figure}

Consider a flow $\Phi$ on a measured topological space $(\Omega,\mu)$. A point $x \in \Omega$ is called {\rm recurrent} for $\Phi$ is for every neighborhood 
$U$ of $x$ and any $T_0 > 0$, there is a time $T > T_0$ such that $\Phi_T(x) \in U$; the flow $\Phi$ itself is {\em recurrent} if $\mu$-almost every
point is recurrent.  
In our setting, the billiard flow $\phi$ is the flow at
constant unit speed  bouncing off at equal
angles upon hitting the rectangular obstacles.    This flow preserve the natural phase volume.

Our first result is that recurrence  satisfies a 0-1 law:
\begin{prop}\label{prop1}
For each ergodic shift invariant measure on $\F$, the wind-tree models in $\FFF^{\Z^2}$ are almost surely recurrent
or almost surely non recurrent.
\end{prop}

The following topological result gives evidence that  wind-tree models are almost surely recurrent.

\begin{theo}\label{typical}
The is a dense $G_{\delta}$ subset $G$ of $\F$ such that the billiard flow is recurrent for every billiard table in $G$ with respect to the natural phase volume.
\end{theo}

A direction $\theta \in \mathbb{S}^1$ is called {\em purely periodic} if all regular orbits are
periodic in this direction are periodic.   The tables in $\F$ are called {\em square tiled} if $\FFF$ is a finite subset
of $\EEE$.  In this case there is a positive integer $Q$, the least common multiple of the denominators of the dimensions
of the obstacles, such that each table $W \in \F$ can be tiled in the standard way (checkerboard tiling) by squares with side length $1/Q$.

\begin{theo}\label{typical1.1}
If $\F$ is square tiled, then there is a dense $G_{\delta}$ subset $G$ of $\F$ such that the billiard flow is recurrent for every billiard table in $G$ with respect to the natural phase volume and for every billiard table in $G$ there is a dense set of purely periodic directions $\theta \subset 
\mathbb{S}^1$.
\end{theo}

In the proofs we will prove the recurrence of certain first return maps.  For this purpose, 
consider a map $F$ on a measured topological space $(\Omega,\mu)$. A point $x \in \Omega$ is called {\rm recurrent} for $F$ if 
for every neighborhood 
$U$ of $x$ there is a time $N > 0$ such that $F^N(x) \in U$; the map $F$ itself is {\em recurrent} if $\mu$-almost every
point is recurrent. 

Next we introduce the first return maps we will use.
Fix a table $W \in \F$ and suppose that $N \ge 1$.
Let $B_N := \{(i,j) \in \Z^2 : |i| + |j| < N\}$ and $A_{N,N_1} := \{(i,j) \in \Z^2 : N \le |i| + |j| < N_1\}$. 
Consider the continuous simple curve $D_N = D_N(W)$ in the billiard table $W$ consisting of the segments of $|x| + |y| = N$ which are in the interior of the table (not in the obstacles)
and the ``outer'' part of the boundary of the obstacles $e$ with centers $(i,j)$ satisfying $|i| + |j| = N$   (see figure 1). 
This curve separates the table into
two parts, the finite (or inner) part and the infinite (or outer) part.  Let $\D_N^-$ consist of the unit vectors with base point in $D_N$ pointing into the finite
part and $\D_N^+$ pointing into the infinite part of the table. Let $\D_N := \D^-_N \cup \D^+_N$ and  $\D = \cup_{N \ge 1} \D_N$. 
Consider the first return maps $f_N : \D_N \to \D_N$ (wherever they are  well defined) and  $f:\D \to \D$.
Clearly this map preserve the phase area if they are well defined.  Note that the singular points  (points whose image, or preimage hits a corner or is tangent to a side of an obstacle) are of measure 0.

\medskip

\begin{figure}[ht]
\begin{minipage}[ht]{0.33\linewidth}
\centering
\begin{tikzpicture}
  \begin{scope}
  [x=1.75cm,y=1.75cm,scale=0.33]
  \foreach \i in {-4,-2,0,...,2,4}
  \foreach \j in {0,2,...,4,6,8}
    \draw[thin] (\i,\j) rectangle +(1,1);
     \draw[very thick,dotted] (0,0) -- (1,0) -- (1,1)-- (2,2) -- (3,2) -- (3,3) -- (4,4) -- (5,4) -- (5,5) -- (4,5) -- (3,6) -- (3,7) -- (2,7) -- (1,8) -- (1,9) -- (0,9) -- (0,8) -- (-1,7) -- (-2,7) -- (-2,6) -- (-3,5) -- (-4,5) -- (-4,4) -- (-3,4) -- (-2,3) -- (-2,2) -- (-1,2) -- (0,1) -- (0,0);

  
  \end{scope}
\end{tikzpicture}
\end{minipage}
\caption{The curve $D_2$ as seen on the table $W_e$ consisting of the obstacle $e$ at such lattice point.}
\end{figure}

The proofs of the results rely on the following lemma.
\begin{lemma}\label{lemma3} The following statements are equivalent.
\begin{enumerate}
\item The wind-tree model $W$ is recurrent 
\item $f : \D \to \D$ is recurrent 
\item there is a positive sequence $\e_n \searrow 0$ and a sequence $N_n \to \infty$
such that $f_{N_n}$ is well defined for at least $(1-\e_n)\%$ of the points in $\D_{N_n}$.
\end{enumerate}
\end{lemma}

\begin{proof}
Clearly (1) implies (2) and (2) implies (3).  We will now show that (3) implies (2).
We claim that if $N_1 < N_2$ and $f_{N_2}$ is well defined almost everywhere
then $f_{N_1}$ is also well defined almost everywhere.  Simply consider the map $f$ induced on the set $\cup_{N \le N_2} \D_N$. This map is 
well defined almost everywhere since $f_{N_2}$ is.  Thus $f_{N_1}$ is recurrent by the Poincare recurrence theorem.  Thus to show that each
$f_N$ is actually defined almost everywhere and thus $f$ is recurrent f it suffices to show that $f_N$ is for infinitely many $N$.  

Note that the map $f$ is invertible.  Consider a set $U \subset \D_{N_n}$
which never recurs to $\D_{N_n}$.  We claim that  $f^j U \cap f^k U = \emptyset$ for all $j > k \ge 0$. If not then by the invertibility of $f$ we would
have $f^{j-k}U \cap U \ne \emptyset$, i.e. the points in $U$ recur to $U \subset \D_{N_n}$, a contradiction.  
This implies that since the set $\D_{N_n}$ are of finite measure, almost every point in $U$ can visit
each set $\D_{N_n}$ only a finite number of times.
Thus we can define for almost every  $x \in U$ a (finite time) $m_{n}(x)$ be the last time the orbit of $x$ visits $\D_{N_{n}}$.
The map $F(x) = f^{m_{n}(x)}(x)$ is a measure preserving map of $U$ into the set of nonrecurrent points in $\D_{N_n}$.  This set
has measure at most $\e_n$. Since this hold for arbitrarily large $n$ we conclude that $U$ is of measure $0$, i.e.\ $f_N$ is well defined
almost everywhere and thus $f$ is recurrent.

Finally we need to show that (2) implies (1).
Consider any small open ball $B$ in the phase space.  Flow each
(non-singular) point in this ball until it hits the set $\D$. Since the ball is open, it has positive phase volume, and its image on the set $\D$
also has positive phase area.   Almost every of these points is $f$ recurrent.  

Fix a nonsingular $x \in B$ such that $x_N := \Phi_t(x) \in \D_N$.  Note that
by transversality and the Fubini theorem  almost every $x \in B$ corresponds to a $f$-recurrent $x_N$.  To conclude the proof we suppose that
$x_N$ is $f$-recurrent and we will show that this implies that  $x$ is $\phi$-recurrent.
Choose a  open neighborhood $U$ of $x$ small enough that for each $y \in U$ there is a $t(y)$ very close to $t$ such that
$\Phi_{t(y)}(y)  \in \D_N$.  Let $U' = \{\Phi_{t(y)}(y) : y \in U\}$.  This is a small neighborhood of $\Phi_t(x)$, and by the above
results there is an (arbitrarily large) $n$ such that $f^n x_N \cap \in U'$. Thus $f^n x_N = \Phi_s(x_N) = \Phi_{s+t}(x) \in U'$ for some large $s$.   
Since  this point $\Phi_{s+t}(x)$ is  in $U'$ is the image $\Phi_{t(y_0)}y_0$ of some $y_0 \in U$.  Thus  $\Phi_{s+t - t(y_0)}(x) = y \in U$
and we conclude that $x$ is recurrent.
\end{proof}


\begin{proofof}{Theorem \ref{typical}} The idea of the proof is simple.  A table in our dense $G_{\delta}$ will have
infinitely many large annuli for which the table has the obstacle $e$ at all lattice sites in the annuli.  The widths of these annuli will increase
sufficiently quickly to guarantee the recurrence.

Fix $\e > 0$ and $N \ge 1$.
Fix a cylinder set in $C = C_{k,N} \in \FFF^{B_N}$, i.e. $C$ is given by specifying the rectangle size (or absence of rectangle) at all the lattice points in $B_N$ (the index $k$ enumerates the finite (or countable) collection of all such cylinder sets).
We consider an $N_1 >> N$ and the cylinder set $C' = C'_{k,N,N_1}$ such that $C' \subset C$ and  for each $c \in C'$ $c_{(i,j)} = e$  
for all $(i,j) \in A_{N,N_1}$.
Consider the table $W_e$. Since it is recurrent, for each fixed $N$, we can choose $N_1 = N_1(N,\e) $ sufficiently large so that on this table
$(1-\e)\%$ of the points in $\D^+_N$ recur to $\D_N$  before leaving $A_{N,N1}$. 
The dynamics for any table in the cylinder $C'$
is identical to the dynamics on the table $W_e$  as long
as it stays in the annulus $A_{N,N_1}$.  Since the wind-tree tables in the cylinder $C'$ agree
with $W_e$ on $A_{N,N_1}$ the $(1-\e)\%$-recurrence hold for all these tables.  

Consider the set $O_{\e} := \cup_{N \ge 1}  \cup_k C'_{k,N,N_1(N,\e)}$.  Since cylinder sets are open this set is open.  Since the 
union is taken over all
cylinder sets it is dense.  Now fix a sequence $\e_n \searrow 0$
and let $G := \cap_n O_{\e_n}$.  Clearly $G$ is a dense $G_{\delta}$ set.  For each table $W \in G$ and for all $n$ there exists $N_n = N_n(W)$
and $k(N_n)$ such that $W \in C'_{k(N_n),N_n,N_1(N_n,\e_n)}$.  This means that for each $n \ge 1$ at least $(1-\e_n)\%$ of the points in the set
$\D_{N_n}$ recur to $\D_{N_n}$.   We apply Lemma \ref{lemma3} to conclude the recurrence.
\end{proofof}

For the proof it is very important that in the annuli the tables agree with a recurrent full occupancy table.
On the other hand for the recurrence the shape of the table in between the annuli is not at all important.  Instead of taking rectangular scatterers, we could choose
circular scatterers, or more byzantine ones (as long as we can define billiard dynamics which preserve the phase volume, for example if they are piecewise $C^1$), disjoint and finitely many in any compact region.

Actually  we can also relax the fact that the tables agree with a recurrent full occupancy table on the annuli.  We can replace this by an
almost agreement in the following sense. The centers of the obstacles could be assumed
to be uniformly distributed in a small open ball around each lattice point.  For a cylinder
set  in which the restriction to an annuli (or another finite set) the centers are very close to the lattice, the 
points in $\D^+_N$ are almost recurrent.  To write the details is technically more complicated since the measure spaces vary in a more
dramatic way than they do in the presented proof. \\

\begin{proofof}{Theorem \ref{typical1.1}} 
The recurrence is a special case of Theorem \ref{typical}. A purely periodic direction is called {\em strongly parabolic} if the phase space 
decomposes into
an infinite number of cylinders isometric to each other.
In \cite{HuLeTr}  it was shown that for the table $W_e$ the set of strongly parabolic directions $\theta$ are dense in 
$\mathbb{S}^1$.  Furthermore  these directions have rational slope ($\pm p/q$ with $p$ and $q$ depending on $\theta$).  
We will show that for any strongly parabolic direction $\theta$ for $W_e$, for any table $W \in G$ a.e.~orbit on this table in the direction
$\theta$ is periodic.  

Fix $\theta$ a strongly parabolic direction for $W_e$ and let $M$ be the common (geometric) length of the cylinders.
Let $n$ be so large that $N_1(N_n,\e_n) - N_n > 2M$ and let $P_n := (N_n + N_1(N_n,\e_n))/2$. On the table $W_e$ all
the cylinders crossing $D_{P_n}$ stay in in $A_{N_n,N_1(N_n,\e_n)}$. Since the table $W$ coincides with $W_e$ on this
set  all regular orbits in $W$ staring on $D_{P_n}$ are periodic.

Now consider any phase point in $W$ with regular orbit starting strictly inside $B_{P_n}$. 
First of all since $D_{P_n}$ consists completely of
periodic orbits this orbit can not reach $D_{P_n}$ without being one of these periodic orbits.  If the orbit does not reach $D_{P_n}$ then it 
stays completely inside $B_{P_n}$.  Since the orbit's slope is rational, bounded and the billiard table is square tiled, the $y$ coordinate
can only take a finite number of values when crossing the lines $x=n/Q$ with $n \in \mathbb{Z}$.  Thus it must visit some
point $(n_0/Q,y_0)$ twice with the same direction. Since the dynamics is invertible the orbit is periodic.
\end{proofof}

\begin{proofof}{Proposition \ref{prop1}}
We use Lemma \ref{lemma3},  a wind-tree model $W \in \F$ is recurrent if and only if there is are sequences $\e_n \searrow 0$ and $N_n \to \infty$ such that $(1-\e_n)\%$ of all points in $\D_{N_n}$ recur to $\D_{N_n}$.  If a wind-tree model is recurrent with the sequences
as above, then there exists $M_n >  N_n$ such that $(1-C\e_n)\%$ of all points in $\D_{N_n}$ recur to $\D_{N_n}$ before hitting $\D_{M_n}$. 
Applying Lemma \ref{lemma3} shows that the converse is  also true.

For each recurrent table, by choosing a subsequence we can suppose that $N_n \ge n$.
Fix a sequence $\e'_n \searrow 0$.  We can assume furthermore, that it satisfies the above condition for this fixed sequence
since if $\e_n \to 0$ faster than $\e'_n$, then it satisfies the condition with $\e'_n$, and if not  then we can
choose a subsequence for which it goes faster.
Summarizing a wind-tree model is recurrent
iff there exists $M_n > N_n \ge n$ such that $(1-C\e'_n)\%$ of all points in $\D_{N_n}$ recur to $\D_{N_n}$ before hitting $\D_{M_n}$.

One needs to specify only a finite part of the table to check this property at a fixed stage $n$, i.e.~if a table $W \in \F$ is recurrent, then all the 
tables in the cylinder set with the obstacles specified to be those of $W$ at
lattice points $(i,j)$ with $|i|+|j| \le M_n$ satisfy this property for this fixed $n$.  Let $O_{N_n,M_n}$ denote the (finite) union of all 
cylinder sets such that this happens at stage $n$.  This is an
open set.  Thus by the above characterization, 
the set of recurrent wind-tree models can be written as $\cap_{n=1}^{\infty} \cup_{M_n > N_n \ge n} O_{N_n,M_n}$ is a Baire 
measurable set.

The notion of recurrence is shift invariant (in the space $\F$).  Thus since the set of recurrent wind-tree models is measurable and invariant
it is of measure 0 or 1 for any invariant measure which is ergodic for the $\Z^2$ shift.
\end{proofof}

\section*{Lorentz gas}
A Lorentz gas is similar to the Ehrenfest
wind-tree model, with the rectangular obstacles replaced by strictly convex ($C^3$) obstacles. A Lorentz gas is said to verify the 
finite horizon condition if  the minimal distance between obstacles is strictly positive and any infinite line intersects infinitely many obstacles with bounded gaps between the intersections.  The long standing conjecture that periodic Lorentz gases are recurrent 
was independently resolved by Conze \cite{Co} and Schmidt \cite{Sc} in the 1990's building on previous results on the hyperbolic structure.  Using the hyperbolic structure 
models, Lenci has shown that Baire typical Lorentz gases with finite horizon are recurrent \cite{Le2}.  
Here we prove that Baire typical 
Lorentz gases are recurrent. Our typical gas will satisfy a weaker property which we call locally finite horizon: the distance between obstacles  is still 
strictly
bounded away from 0 and every infinite line intersects infinitely many obstacles,
however the gaps between the intersections are not necessarily bounded. 

Rather than try to state a general result, we give two examples.  Generalizations to other situations should be clear.\\
1) Let $\mathbb{T}$ denote the triangular lattice.
Fix $e$ and convex open set with $C^3$ boundary (for example a ball), such that if we place the obstacle $e$ at
each lattice site, then the corresponding infinite table satisfies the finite horizon condition.
Let $0$ denote the absence of an obstacle. Then the set of Lorentz billiard tables
we consider is $\{0,e\}^{\mathbb{T}}$, note that since we are allowing empty cells, these tables do not have finite horizon, nor
apriori locally finite horizon.

2) Consider the $\Z^2$ lattice.  Here we let $e$ denote the obstacle consisting of the union of 5 convex open sets with 
$C^3$ boundary (again for examples balls). 
The  5 convex sets are chosen
so the table consisting of the obstacle $e$ at each lattice site is of finite horizon.  Again $0$ denotes the absence of an
obstacle.  We consider the set of Lorentz gases $\{0,e\}^{\Z^2}$. These tables, like those above, do not necessarily have
finite horizon, nor locally finite horizon.

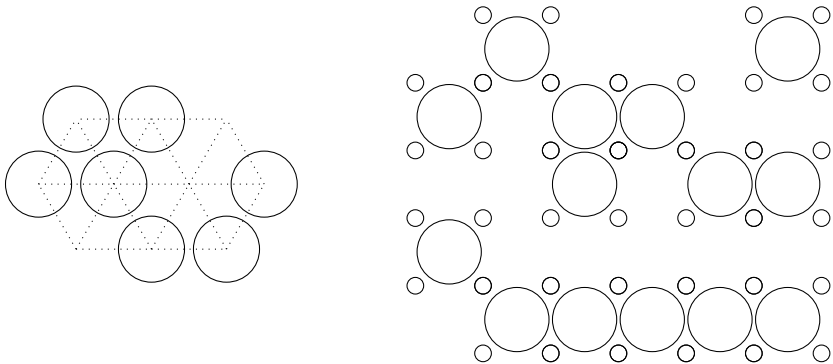
\begin{figure}[ht]
\begin{minipage}[ht]{0.48\linewidth}
\centering
\begin{tikzpicture} [scale=1]
\coordinate [] (A1) at (0,0); 
\coordinate [] (A2) at (1,0); 
\coordinate [] (A3) at (2,0);
 \coordinate [] (A4) at (3,0);

\draw [dotted] (A1) -- (A2);
\draw [dotted] (A2) -- (A3);
\draw [dotted] (A3) -- (A4);

\node (C1) [circle through=(A2)] at (A1) {}; 
\node (C2) [circle through=(A1)] at (A2) {};
\node (C3) [circle through=(A2)] at (A3) {}; 
\node (C4) [circle through=(A3)] at (A4) {}; 

\coordinate [] (B1) at (intersection 2 of C1 and C2);
\coordinate [] (B2) at (intersection 1 of C3 and C2);
\coordinate [] (D1) at (intersection 1 of C1 and C2);
\coordinate [] (D2) at (intersection 2 of C3 and C2);
\coordinate [] (B3) at (intersection 2 of C3 and C4);
\coordinate [] (D3) at (intersection 1 of C3 and C4);

\draw [dotted] (A1) -- (B1);
\draw [dotted] (A2) -- (B1); 
\draw [dotted] (A2) -- (B2);
\draw [dotted] (A3) -- (B2);
\draw [dotted] (A3) -- (B3);
\draw [dotted] (A4) -- (B3);
 
\draw [dotted] (B2) -- (B1);
\draw [dotted] (D1) -- (D2);
\draw [dotted] (D2) -- (D3);
\draw [dotted] (B2) -- (B3);

\draw [dotted] (D1) -- (A1); 
\draw [dotted] (D1) -- (A2);
\draw [dotted] (D2) -- (A2);
\draw [dotted] (D2) -- (A3);
\draw [dotted] (D3) -- (A4);
\draw [dotted] (D3) -- (A3);

\draw (B1) circle (12.5pt);
\draw (A2) circle (12.5pt);
\draw (A1) circle (12.5pt);
\draw (A4) circle (12.5pt);
\draw (B2) circle (12.5pt);
\draw (D3) circle (12.5pt);
\draw (D2) circle (12.5pt);

\end{tikzpicture}
\end{minipage}
\begin{minipage}[ht]{0.48\linewidth}
\centering

\begin{tikzpicture}[scale=0.45]
\draw (2,0) circle (27pt);
\draw (3,1) circle (7pt);
\draw (3,-1) circle (7pt);
\draw (1,1) circle (7pt);
\draw (1,-1) circle (7pt);



\foreach \i in {2,6,8}
 {   \draw (\i,4) circle (27pt);
      \draw (\i+1,5) circle (7pt);
\draw (\i+1,3) circle (7pt);
\draw (\i-1,5) circle (7pt);
\draw (\i-1,3) circle (7pt);}

\foreach \i in {4,6,8,10,12}    
\foreach \j in {-2}
{ \draw (\i,\j) circle (27pt);
     \draw (\i+1,\j+1) circle (7pt);
\draw (\i+1,\j-1) circle (7pt);
\draw (\i-1,\j-1) circle (7pt);
\draw (\i-1,\j+1) circle (7pt);}

\foreach \i in   {6,10,12} 
\foreach \j in {2} 
  { \draw (\i,\j) circle (27pt);
     \draw (\i+1,\j+1) circle (7pt);
\draw (\i+1,\j-1) circle (7pt);
\draw (\i-1,\j-1) circle (7pt);
\draw (\i-1,\j+1) circle (7pt);}

     \foreach \i in   {4,12} 
\foreach \j in {6} 
  { \draw (\i,\j) circle (27pt);
     \draw (\i+1,\j+1) circle (7pt);
\draw (\i+1,\j-1) circle (7pt);
\draw (\i-1,\j-1) circle (7pt);
\draw (\i-1,\j+1) circle (7pt);}
 \end{tikzpicture}
    \end{minipage}%
    \caption{Random Lorentz gases with triangular and square lattices}

\end{figure}

The proof  of the following theorem
is essentially identical to the proof of Theorem 2 and will be omitted.

\begin{theo}\label{typical1}
The is a dense $G_{\delta}$ subset $G$ of each of the above two example such that the billiard flow is recurrent for every billiard table in $G$ with respect to the natural phase volume.
\end{theo}

All the tables in the dense $G_{\delta}$ will have a locally finite horizon.
It would be interesting to investigate if this implies ergodicity like in the finite horizon case (see \cite{Le1},\cite{Le2}).

\bigskip

Acknowledgements:  Many thanks to Marco Lenci for useful comments.

\end{document}